\documentclass[11pt]{amsart}

\usepackage{amsmath}
\usepackage{amssymb}
\usepackage{epsfig}
\usepackage{psfrag}

\usepackage{amsfonts}
\usepackage[dvipsnames]{xcolor}

\usepackage{ytableau,tikz}

\expandafter\def\expandafter\normalsize\expandafter{%
    \normalsize
    \setlength\abovedisplayskip{7pt}
    \setlength\belowdisplayskip{7pt}
    \setlength\abovedisplayshortskip{5pt}
    \setlength\belowdisplayshortskip{5pt}
}

\newcommand{\excise}[1]{}

\newtheorem{thm}{Theorem}[section]
\newtheorem{lemma}[thm]{Lemma}

\newtheorem{cor}[thm]{Corollary}
\newtheorem{prop}[thm]{Proposition}
\newtheorem{conj}[thm]{Conjecture}
\newtheorem{conv}[thm]{Convention}

\newtheorem{ex}[thm]{Example}
\newtheorem{rem}[thm]{Remark}
\newtheorem{Warn}[thm]{Caution}

    {\end{Example}}
    {\end{Remark}}

\def\rr{\mathbb R}

\def\ov{\overline}

\def\la{\lambda}
\def\ga{\gamma}

\def\al{\alpha}
\def\be{\beta}

\def\vp{\varphi}
\def\vk{\varkappa}

\def\ssu{\subset}

\def\<{\langle}
\def\>{\rangle}

\def\0{{\mathbf 0}}

\def\sgn{{\rm sgn}}

\def\.{\hskip.06cm}
\def\ts{\hskip.03cm}

\def\SM{{\mathcal{K}}}
\def\MN{ {\text {\rm MN} } }

\def\size{{\text{size}}}
\def\sign{{\rm sign}}

\def\KP{{\textsc{KP}}}

\def\P{{\rm{\textsf{P}}}}
\def\FP{{\rm{\textsf{FP}}}}
\def\SP{{\rm{\textsf{\#P}}}}
\def\FP{{\rm{\textsf{FP}}}}
\def\NP{{\rm{\textsf{NP}}}}
\def\KN{{\textsc{Kron}}}

\def\CP{{\textsc{CharP}}}
\def\GP{{\rm{\textsf{GapP}}}}

\def\pa{P}

\def\gi{{\mathcal{X}}}

\def\Poly{{\text{\rm Poly}}}
\def\LR{{LR}}

\def\shp{\varphi}
\newcommand{\shortp}[1]{\shp(#1)}

\begin{document}
\title[Complexity of Kronecker coefficients]{On the complexity of computing Kronecker coefficients}

\author[Igor~Pak]{ \ Igor~Pak$^\star$}

\author[Greta~Panova]{ \ Greta~Panova$^\star$}

\date{\today}


\thanks{\thinspace ${\hspace{-.45ex}}^\star$Department of Mathematics, UCLA, Los Angeles, CA 90095, \ts
\texttt{\{pak,panova\}@math.ucla.edu}}

\begin{abstract}
We study the complexity of computing \emph{Kronecker coefficients}
$g(\la,\mu,\nu)$.  We give explicit bounds in terms of the number of
parts~$\ell$ in the partitions, their largest part size $N$ and the
smallest second part~$M$ of the three partitions. When $M = O(1)$,
i.e.~one of the partitions is~\emph{hook-like}, the bounds are linear
in~$\log N$, but depend exponentially on~$\ell$. Moreover, similar bounds hold even when $M=e^{O(\ell)}$.
By a separate argument, we show that the positivity of Kronecker
coefficients can be decided in $O(\log N)$ time for a bounded
number~$\ell$ of parts and without restriction on~$M$.
Related problems of computing Kronecker coefficients when one partition
is a hook, and computing characters of~$S_n$ are also considered.
\end{abstract}

\maketitle



\section{Introduction and main results}\label{intro}

\noindent
The study of \emph{Kronecker coefficients} $g(\la,\mu,\nu)$ of the symmetric
group~$S_n$ has rare qualities of being  classical, highly technical,
and largely mysterious.
The area was initiated by Murnaghan 75 years ago~\cite{Mur}, and continued
to be active for decades, with scores of interesting connections to other
areas. Despite a large body of work on the Kronecker coefficients, both
classical and very recent, it is universally agreed that ``frustratingly little
is known about them''~\cite{Bur}.
The problem of finding a combinatorial interpretation for
$g(\la,\mu,\nu)$, can be restated as whether computing the coefficients
is in~\SP. It remains a major open problem, one of the oldest unsolved
problems in Algebraic Combinatorics.

More recently, the interest in computing Kronecker coefficients has
intensified in connection with  \emph{Geometric Complexity Theory},
 pioneered recently as an approach to the $\P$~vs.~$\NP$~problem
(see~\cite{Mul2,MS,Reg}).
With Valiant's theory of determinant computations as its starting point,
their approach relies, among other things,
on the (conjectural) ability to decide in polynomial time
the positivity of Kronecker coefficients and their plethystic generalizations.
Envisioned as a far reaching mathematical program requiring over 100~years to
complete~\cite{For}, this approach led to a flurry of activity in an attempt
to understand and establish some critical combinatorial and computational
properties of Kronecker coefficients (see~\cite{BOR,BI,CDW,Ike,Mul1}).  This paper
is a new advance in this direction.

While we present several algorithmic and complexity results, they are centered
around a single unifying problem.  We are trying
to understand what exactly makes the Kronecker coefficients hard to compute.
Since the problem is \SP-hard in general (see~\cite{BI}), a polynomial time algorithm for deciding positivity
is unlikely to exist. On the other hand, the problem can be viewed as a
generalization of LR~coefficients $c_{\be\ga}^\al$, another
\SP-complete problem.   In view of a nice geometric interpretation of the latter,
it can be computed using \emph{Barvinok's algorithm} in polynomial
time for any fixed~$\ell$ (see~$\S$\ref{ss:fin-barv}).  We show that
a similar analysis applies to Kronecker coefficients.  In other words,
it is not the large part sizes that make an obstacle, but the number~$\ell$
of parts in the partitions.

\smallskip

\noindent
Our main result is the following theorem, which introduces a new
parameter~$M$, the smallest number among the second parts of the three partitions.
Our complexity bound is general, but is especially sharp for triples of
partitions when one is \emph{hook-like}.  We state the theorem here
in a somewhat abbreviated form, as we postpone the definitions and details.

%

\medskip

\noindent
{\bf Main Theorem~\ref{thm:nu_bounded}.} \.
{\it Let $\la,\mu,\nu \vdash n$ be partitions with lengths
$\ell(\la),\ell(\mu)$, $\ell(\nu) \le \ell$, the largest parts
$\la_1,\mu_1,\nu_1 \le N$, and $\nu_2 \le M$.  Then the
Kronecker coefficients $g(\la,\mu,\nu)$ can be computed in time
$$O(\ell \log N)+ (\ell\log M)^{O(\ell^3\log\ell)}\, .$$
}

\medskip

To illustrate the result, consider several special cases of the theorem.
First, when~$\ell$ is fixed, the recent breakthrough in~\cite{CDW}
gives a polynomial time algorithm for computing Kronecker coefficients.
The main theorem in this case is Main Lemma~\ref{prop:g_poly} which
gives explicit bounds on the dependence of~$\ell$. 
   Curiously, even for a simpler problem of
deciding \emph{positivity} $g(\la,\mu,\nu)>0$ this gives the best
known general bound.  Our Theorem~\ref{thm:fixed} uses the
\emph{semigroup property} of Kronecker coefficients to give a
surprisingly powerful \emph{linear} bound, but without giving explicitly the dependence
on~$\ell$.

Second, when we have $\la_2,\mu_2,\nu_2 \le M$ and $N$ is large compared to~$(M\ell)$,
the Kronecker coefficients stabilize to \emph{reduced Kronecker coefficients},
which generalize LR~coefficients and are believed to be easier to compute
(see~$\S$~\ref{ss:fin-red}).  Our Main Theorem lends further support to this
conjecture.

Finally, when $M=O(1)$ the theorem gives a new type complexity bound
of computing $g(\la,\mu,\nu)>0$ when $\nu$ is \emph{hook-like}.
This may seem surprising, as already the case of \emph{hooks} (i.e.\/ when $M=1$),
received considerable attention in the literature (see e.g.~\cite{Las,Rem,Ros}).
There, even in the simplest cases, the resulting formulas for Kronecker coefficients
seem rather difficult, and the recent combinatorial interpretation by Blasiak
unsuitable for efficient computation~\cite{Bla}.  Curiously, we use Blasiak's
combinatorial interpretation of Kronecker coefficients to show that
computing $g(\la,\mu,\nu)$ is in~\SP~when $\nu$ is a hook
(Theorem~\ref{thm:hooks}).

\medskip

\noindent
{\bf Corollary 1.1} \. {\it In the notation of the Main Theorem, suppose
$$
\log M \ts, \ts \ell \. = \. O\left(\frac{(\log \log N)^{1/3}}{(\log\log\log N)^{2/3}}\right)\ts.
$$
Then there is a polynomial time algorithm to compute \ts $g(\la,\mu,\nu)$. }

\medskip

The proofs are based on two main tools.  The first is the \emph{Reduction Lemma}
(Lemma~\ref{l:reduction}), which implies that when $\nu_2$ is small,
we either immediately have $g(\la,\mu,\nu)=0$, or else there are
partitions $\shortp{\la}$, $\shortp{\mu}$, $\shortp{\nu}$ of size
$O(\ell^3\nu_2)$, such that
$g(\la,\mu,\nu)=g(\shortp{\la},\shortp{\mu},\shortp{\nu})$.
In other words, we reduce the problem from binary input to
unary input, and apply Lemma~\ref{prop:g_poly} to the latter case.

The second tool is the \emph{Main Lemma}~\ref{prop:g_poly} mentioned above,
which gives an effective bound on the complexity
of computing Kronecker coefficients without any restrictions on~$M$.
It coincides with the Main Theorem~\ref{thm:nu_bounded}
when~$M=N$, and  states that the Kronecker coefficients can be
computed in time $\Poly\bigl((\ell\log N)^{\ell^3\log\ell}\bigr)$.
This is achieved by separating the algebraic and complexity parts;
the latter is reduced to counting integer points in certain $3$-way
statistical tables via Barvinok's algorithm
(see $\S$\ref{ss:barv} and~$\S$\ref{ss:fin-barv}).

\smallskip

The rest of this paper is structured as follows.  We begin
with basic definitions in Section~\ref{s:def} and proceed to
state our new complexity results in Section~\ref{s:complexity}.
In Section~\ref{s:hooks}, we discuss Blasiak's combinatorial
interpretation and its implications.  This section is largely
separate from the rest of the paper and uses some background
in Algebraic Combinatorics.

The main results of this paper, notably the Main Lemma and the
Reduction Lemma are proved
in Section~\ref{s:complexity_kp}.  We follow with two
short sections~\ref{ss:fixed_length} and~\ref{s:char}
discussing the case of bounded~$\ell$ and the
complexity of computing the characters of~$S_n$,
respectively.
Namely,  we prove that the problem of deciding
whether $\chi^{\la}[\nu]=0$ is \NP--hard, extending earlier
easy results by Hepler~\cite{Hel} (see~$\S$\ref{ss:fin-char}
for the connection with Kronecker coefficients).  We conclude with final
remarks and open problems in Section~\ref{s:final}.

\bigskip

\section{Definitions and background}\label{s:def}

We briefly remind the reader of basic definitions, standard notations
and several claims which will be used throughout the paper.
For more background on the representation theory of the symmetric
group and related combinatorics, see e.g.~\cite{Mac,Sag,Sta}.

\subsection{Partitions and characters}\label{ss:def-part}
Let $\la=(\la_1,\la_2,\ldots) \vdash n$ be a partition of~$n$,
and let $P_n$ denote the set of partitions of~$n$.
Denote by~$\la'$ the conjugate partition of $\la$.
Denote by $\ell(\la)=\la_1'$ the number of parts in~$\la$.
We use Young diagram $[\la]$ to represent a partition~$\la$.
Partitions $(n-k,1^k)$ are called \emph{hooks}; partitions $(n-k,k)$
with two parts will also play a major role.
We also define the union and intersection of two partitions as the
union or intersection of their Young diagrams.  In other words,
$\pi = \la \cup \mu$ mean that $\pi_i = \max(\la_i,\mu_i)$, and
$\rho= \la \cap \mu$ means that $\rho_i= \min(\la_i,\mu_i)$.
Denote by $\la+\mu$ the partition $(\la_1+\mu_1,\la_2+\mu_2,\ldots)$.

We denote by $\chi^\la$, $\la\vdash n$, the irreducible character of the
symmetric group~$S_n$, and $\chi^\la[\mu]$ be its value $\chi^\la(u)$
on any permutation~$u$ of cycle type~$\mu$.
For a general character~$\eta$, the multiplicity of~$\chi^\la$ in $\eta$
is given by the scalar product:
$$
c(\chi^\la,\eta) \. = \. \frac{1}{n!} \. \sum_{u\in S_n}
\. \chi^\la(u) \ts \eta(u)\ts.
$$
We use \. ``$\sign$'' \. to denote character corresponding to partition~$(1^n)$.

The characters can be computed by the \emph{Murnaghan--Nakayama rule}
(see e.g.~\cite{Sag,Sta}).  Briefly, it says that
$$\chi^\la[\mu] \. = \. \sum_{B \ts \in\ts \MN^\la_\mu} \. (-1)^{ht(B)-\ell(\mu)},$$
where $\MN^\la_\mu$ is the set of all border-strip tableaux of shape $\lambda$ and type $\mu$ and $ht(B)$ is the sum of the number of rows (height) in each border-strip of~$B$. A \emph{border-strip} is a skew connected Young diagram which does not contain a $2\times 2$ square of boxes. A \emph{border-strip tableaux} of shape $\la$ and type $\mu$ is a filling of the Young diagram of $\la$ with $\mu_1$ integers~1, $\mu_2$ integers~2, etc., such that the entries along each row and down column are weakly increasing, and such that all squares with the same number form a border-strip.

\smallskip

\subsection{Kronecker coefficients}\label{ss:def-kron}
We use $\chi \otimes \eta$ to denote the tensor product of characters.
The \emph{Kronecker coefficients} $g(\la,\mu,\nu)$, where $\la,\mu,\nu \vdash n$
are given by
$$
\chi^\la \otimes \chi^\mu \, = \, \sum_{\nu\vdash n} \. g(\la,\mu,\nu) \ts \chi^\nu\ts.
$$
It is well known that
$$
g(\la,\mu,\nu) \, = \, \frac{1}{n!} \. \sum_{\omega \in S_n} \. \chi^\la(\omega)\chi^\mu(\omega)\chi^\nu(\omega)\ts.
$$
This implies that Kronecker coefficients have full $S_3$ group of symmetry:
$$
g(\la,\mu,\nu) \, = \, g(\mu,\la,\nu) \, = \, g(\la,\nu,\mu) \, = \, \ldots 
$$
In addition, recall that $\chi^\la \otimes \sign = \chi^{\la'}$.  This implies
$$
g(\la,\mu,\nu) \, = \, g(\la',\mu',\nu) \, = \, g(\la',\mu,\nu') \, = \, g(\la,\mu',\nu')\ts.
$$

\smallskip

\subsection{Symmetric functions}\label{ss:def-LR}
We denote by $h_n$ and $h_\la=h_{\la_1}h_{\la_2}\cdots$
the \emph{homogeneous} symmetric functions, and by $s_\la$
the \emph{Schur functions} (see e.g.~\cite{Mac,Sta}).
The \emph{Littlewood--Richardson}~(LR) \emph{coefficients}
are denoted by $\LR(\la,\mu,\nu) = c^\la_{\mu,\nu}$, where
\. $|\la|\ts =\ts |\mu|+|\nu| \ts = \ts n$.  They are given by
$$
s_\mu \ts \cdot \ts s_\nu \, = \, \sum_{\la\vdash n} \. c^\la_{\mu,\nu} \ts s_\la\..
$$
The integers $c^\la_{\mu,\nu}$ have a combinatorial interpretation in term
of certain semistandard Young tableaux (see~\cite{Sag,Sta}) and \emph{BZ triangles}
(see e.g.~\cite{PV}).

Define the \emph{Kronecker product} of symmetric functions as follows:
$$
s_\mu \ts * \ts s_\nu \, = \, \sum_{\la\vdash n}  \. g(\la,\mu,\nu) \. s_\la
\..
$$

The following \emph{Littlewood's identity} (see~\cite{Lit}) is crucial for our study:
\begin{equation}\label{Littlewood}
s_\la * (s_\tau s_\theta) \. = \.
\sum_{\alpha \vdash |\tau|, \ts \beta \vdash |\theta|} \. c^\la_{\alpha \beta} (s_\alpha*s_\tau)(s_\beta*s_\theta)\..
\end{equation}
We also need the \emph{generalized Cauchy's identity} (see~\cite[Ex I.7.10]{Mac} and~\cite[Ex 7.78]{Sta})

\begin{align}\label{extended_identity}
\sum_{\lambda,\mu,\nu} \. g(\lambda,\mu,\nu)\ts s_{\lambda}(x)\ts s_{\mu}(y)\ts s_{\nu}(z) \, = \,
\prod_{i,j,\ell} \. \frac{1}{1-x_iy_jz_{\ell}}\,.
\end{align}
Given a power series $F=a_0 + a_1t+a_2t^2+\ldots$\ts, denote by $[t^i]\ts F$ the coefficient~$a_i$.  Similarly,
when~$F$ is a symmetric function and $A$ is a Schur function, denote by $[A]\ts F$ the coefficient of~$A$
in the expansion of~$F$ in the linear basis of Schur functions. By a slight abuse of notation, we use
$[A]\ts F$ for other bases of symmetric functions as well.


\subsection{Semigroup property}\label{ss:semigroup}

The triples $(\lambda,\mu,\nu)$ for which $g(\lambda,\mu,\nu)>0$ form a semigroup in the following sense.

\begin{thm}[Semigroup property]  \label{t:semi}
Suppose $\la,\mu,\nu,\al,\be,\ga$ are partitions of~$n$, such that \\
$\ts g(\la,\mu,\nu) >0\ts $ and $\ts g(\al,\be,\ga)> 0$. Then
$\ts g(\la+\al,\mu+\be,\nu+\ga)\ts >\ts 0$.
\end{thm}

This leads to the following definition of the \emph{Kronecker semigroup}.
Let $\SM_\ell$ denote the set of triples of partitions $(\la,\mu,\nu)$ written as vectors
$$
(\la_1,\ldots,\la_\ell, \. \mu_1,\ldots,\mu_\ell, \. \nu_1,\ldots,\nu_\ell)\.,
$$
such that $g(\la,\mu,\nu)>0$.  The theorem implies that $\SM_\ell$ is a
semigroup under addition.

\begin{cor}\label{c:semi}
The Kronecker semigroup $\SM_\ell$ is finitely generated.
\end{cor}

Both results are proved in~\cite{CHM} (see also $\S$\ref{ss:fin-semi}
and~\cite[$\S$4.4]{Ike}).
We use them crucially in Section~\ref{ss:fixed_length}.

\subsection{Barvinok's algorithm}\label{ss:barv}

Let $P \ssu \rr^d$ be a convex polytope given by a system of linear
equations and inequalities over integers.  Denote by $L$ the size of
the input.

\begin{thm}[Barvinok]  \label{t:barv}
For every fixed~$d$, there is a polynomial algorithm computing
the number of integer points in~$P$.  Furthermore, for general~$d$,
the algorithm works in $L^{O(d\log d)}$ time.
\end{thm}

The original algorithm by Barvinok required $L^{O(d^2)}$ time,
which was subsequently reduced to that in the theorem.  We refer
to~\cite{Bar,Bar-alg} for the proof, detailed surveys and
further references (see also~\cite{DHTY,DK}).\footnote{Occasionally,
this complexity is reported as~$L^{O(d)}$, but it seems a more
careful accounting gives the bound as in the theorem (A.~Barvinok,
personal communication). }

\bigskip

\section{Complexity problems}\label{s:complexity}

\subsection{Decision problems}\label{ss:decision_problems}
We are interested in deciding whether Kronecker coefficients
$g(\la,\mu,\nu)$ are strictly positive.


\bigskip

\noindent
{\sc Positivity of Kronecker coefficients (\KP):}

\smallskip

\noindent
{\bf Input:} \, Integers $N, \ell$, partitions $\la=(\la_1,\ldots,\la_\ell)$, $\mu=(\mu_1,\ldots,\mu_\ell)$,

\noindent
 $\nu=(\nu_1,\ldots,\nu_\ell)$, where $0\le \la_i, \mu_i, \nu_i \le N$, and $|\la|=|\mu|=|\nu|$.

\smallskip

\noindent
{\bf Decide:} \, whether $g(\la,\mu,\nu)>0$.

\bigskip

Recall that the two ways to present the input: in \emph{binary} and in \emph{unary}.
The difference is in the input size, denoted $\size(\la,\mu,\nu)$:  in the binary case
we have  $\size(\la,\mu,\nu) = \Theta(\ell \log N)$, and in the unary case
$\size(\la,\mu,\nu) = \Theta(\ell N)$.   Throughout the paper we
assume the input is in binary, unless specified otherwise.
The problem then becomes a well known conjecture:

\begin{conj}[Mulmuley]\label{conj:decision}
{\KP}~$\in$~{\P}.
\end{conj}

Note that it is not even known whether \KP~$\in$~\NP,
except in a few special cases, such as when one of the
input partitions is a hook.
This case is considered in Section~\ref{s:hooks}. Here we consider various subproblems like the case when $\ell$ is fixed, denoted \KP($\ell$), and the case when one partition is a hook, denoted \KP(hook).

%

\subsection{Counting problems.}
There are analogous problems about computing the
exact values of the coefficients mentioned above.
\smallskip

\noindent
{\sc Kronecker coefficients (\KN):}

\smallskip

\noindent
{\bf Input:} \, Integers $N, \ell$, partitions $\la=(\la_1,\ldots,\la_\ell)$, $\mu=(\mu_1,\ldots,\mu_\ell)$,

\noindent
 $\nu=(\nu_1,\ldots,\nu_\ell)$, where $0\le \la_i, \mu_i, \nu_i \le N$, and $|\la|=|\mu|=|\nu|$.

\smallskip

\noindent
{\bf Compute:} \,the Kronecker coefficient $g(\la,\mu,\nu)$.
\smallskip

Analogously to the \KP \ case, we consider also the subproblems \KN($\ell$) when $\ell$ is fixed and \KN(hook) when one partition is a hook.

\medskip

The main complexity result in the area is the following recent theorem.

\begin{thm}[B\"urgisser--Ikenmeyer]\label{thm:gapp}
\KN~$\in$~\GP.
\end{thm}

Here \GP~is a class of functions obtained as a difference
of two functions in~\SP.  We give a different proof of the theorem
in Section~\ref{s:complexity_kp} (cf.~\cite{CDW}).
We conclude with two more conjectures by Mulmuley~\cite{Mul1}.

\begin{conj}[Mulmuley]  \label{conj:sp}
\KN~$\in$~\SP.
\end{conj}

\begin{conj}[Mulmuley]  \label{conj:sp-pol}
The Kronecker coefficient $g(\la,\mu,\nu)$ is equal to
the number of integer points in convex polytope $P(\la,\mu,\nu)$
with a polynomial description.
\end{conj}

This conjecture is the counting version of Conjecture~\ref{conj:decision}.
It extends the classical result by Gelfand and Zelevinsky, expressing
LR~coefficients as the number of integer points in convex polytopes
(see e.g.~\cite{Zel})

\smallskip
The main result of the paper is the following theorem
(the proof is in Section~\ref{s:complexity_kp}).

\begin{thm}[Main Theorem] \label{thm:nu_bounded}
Consider the problem \KN, where the input is integers $N,\ell$ and  partitions
$\lambda=(\la_1,\la_2,\ldots,\la_\ell), \mu=(\mu_1,\mu_2,\ldots,\mu_\ell),\nu=(\nu_1,\ldots,\nu_\ell)$,
such that $\la_1,\mu_1,\nu_1 \leq N$ and $|\la|=|\mu|=|\nu|$.
Suppose that $\nu_2 \leq M$. Then $g(\la,\mu,\nu)$ can be computed
in time
$$O(\ell\ts\log N) \. + \. (\ell\ts \log M)^{O(\ell^3\log\ell)}.$$
\end{thm}

\medskip

\begin{cor}[\cite{CDW}]\label{cor:fixed}
We have \KP$(\ell)$~$\in$~\P~and \KN$(\ell)$~$\in$~\FP, for every fixed~$\ell$.
\end{cor}

 Here {\FP} is a class of functions computable in polynomial time,
a counterpart of {\P}~for decision problems. Note that
in Section~\ref{ss:fixed_length}, we use the semigroup property
to prove the first part of the corollary.

\bigskip

\section{The case of a hook}\label{s:hooks}

Here we consider separately the complexity of \KP(hook)~and \KN(hook)~when one of the partitions involved is a hook.

Let $\nu=(n-t,1^t)$ be a hook shape, and $\la,\mu\vdash n$, such that $\ell(\lambda),
\ell(\mu) \leq k$ and $\la_1,\mu_1\leq N$.  Theorem~\ref{thm:nu_bounded} implies the following result in this case.

\begin{cor}\label{cor:hooks}
Let $\nu=(n-t,1^t)$ be a hook, $\la_1,\mu_1\le N$ and $\ell(\la),\ell(\mu)\le k$.
Then $g(\la,\mu,\nu)$ can be computed in time
$$
O\left(k^2\log N \right)\. + \. k^{O(k^6\log k)} \. .
$$
\end{cor}

\begin{proof}
Recall  that when $\ell(\lambda)\ell(\mu) < \ell(\nu)$ we have $g(\lambda,\mu,\nu)=0$, see~\cite{Dvir}. Thus, when $t>k^2$ this implies  $g(\lambda,\mu,\nu)=0$. On the other hand, when the height of the hook $t \leq k^2$, apply Theorem~\ref{thm:nu_bounded} with $\ell = k^2$ and $M=1$, to obtain the result.
\end{proof}

\begin{rem}{\rm
Note that Lemma~\ref{l:contingency} and subsequently Lemma~\ref{prop:g_poly} below
can be easily modified for partitions of different lengths.  This more careful
analysis in the proofs of lemmas can reduce the exponent in Corollary~\ref{cor:hooks}
from $k^6$ to~$k^4$.  We omit this improvement for the sake of clarity.  }
\end{rem}

We use the recent combinatorial interpretation by Blasiak in~\cite{Bla}
as outlined below, to prove the following result.
\begin{thm}\label{thm:hooks}
We have \KP{\rm (hook)}~$\in$~\NP~and \KN{\rm (hook)}~$\in$~\SP.
\end{thm}
We consider barred and unbarred entries $\ov 1, \ov 2, \ldots, 1, 2, \ldots$,
which we use to fill a Young tableau.  Such tableau is called \emph{semi-standard}
if the entries are weakly increasing in both rows and columns, with no two equal
barred numbers in the same row, and no two equal unbarred numbers in the same
column.  The \emph{content} of a tableau~$A$ is a sequence
$(m_1,m_2,\ldots)$, where $m_r$ is the total number of~$r$ and $\ov r$ entries in~$A$.
Two orders are considered:
$$\text{\it natural order}: \qquad \ \ov 1 \. < \.  1 \. <  \ov 2 \. < \.  2 \. < \. \ldots \ \ \, \text{and}
$$
$$\text{\it small barred order}: \qquad \ov 1 \. \prec  \.  \ov 2 \. \prec  \. \ldots \.
\prec \. 1 \. \prec  \.  2 \. \prec \. \ldots
$$
An example of two tableaux with different orders but the same shape and content
is given in Figure~\ref{f:semi}.
\begin{figure}[hbtp]
   \includegraphics[scale=0.38]{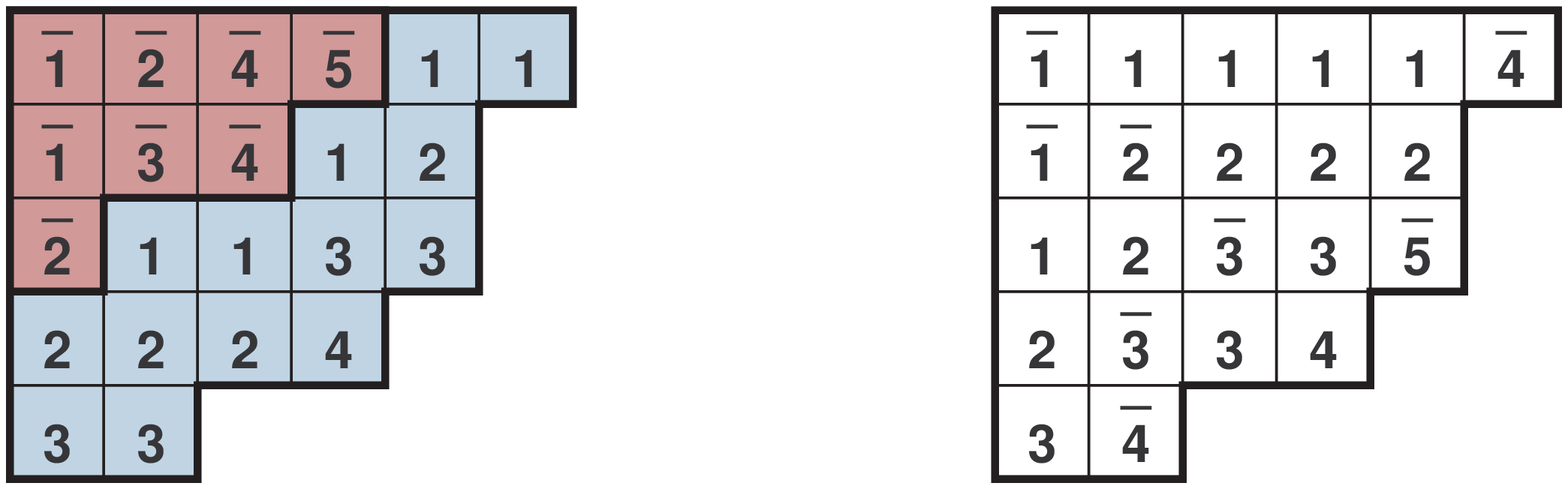}
   \caption{Semi-standard Young tableaux of shape $(6,5^2,4,2)$ with small barred and natural order,
   of the same content $(7,6,5,3,2)$.}
   \label{f:semi}
\end{figure}

There is a natural \emph{tableau switching} bijection $A \leftrightarrow \vk(A)$ between semi-standard tableaux
with natural and small bar ordering, preserving the shape and content.  The idea is to make a number
of jeu-de-taquin slides exchanging barred and unbarred entries in order to convert the tableau from one order to the other. Specifically, a jeu-de-taquin slide is the following local operation: given an ``out-of-order'' entry $c$, i.e. such that the element to its left $a$ and/or the element above it $b$ is larger, then we exchange $c$ with the larger among $a,b$:
\ytableausetup{boxsize=2.5ex}
$$\ytableaushort{\none b,ac} \quad \text{ with $c>a$ and/or $c>b$ } \longrightarrow \begin{cases} \ytableaushort{\none c,ab} \quad , & \text{ if $b> a$ or $b=a$ and are unbarred,} \\[0.2in] \ytableaushort{\none b,ca} \quad , & \text{ if $a>b$ or $a=b$ and are barred.}\end{cases} $$
Tableau-switching is then the process in which we start with a tableau in the small-bar order and sort it into the natural order, by applying the jeu-de-taquin described above starting with the smallest and left-most unbarred entry which is out-of-order with respect to the natural ordering and moving it ``up'' until both entries above and left of it are smaller or equal. Then we continue with the next smallest left-most out-of-order entry and so on until the tableau is a natural SSYT. It is a well known that this entire process is well-defined (see~\cite{Bla,BSS}).

Note that one can view a small bar order as a pair of two semi-standard Young tableaux~$A_1$
and~$A_2$, one of shape $\la/\nu$ and one of shape~$\nu'$, respectively. Here $A_1$ is the subtableau consisting of the unbarred entries of $A$, and $A_2$ is the conjugate of the subtableau of barred entries of $A$. Denote by $w(A)$
the word obtained by reading $A_2$ right to left, top row to bottom row, concatenated with
the word obtained by reading $A_1$ top to bottom, right to left.  For example, for~$A$
as in the figure, we have
$$
w(A) \. = \. 11.21.3311.4222.33..2.134.1245\ts.
$$
By $m_i(w)$ denote the number of $i$-s in~$w$.
Word~$w$ is called a \emph{ballot sequence} if in its every prefix~$w'$
we have \. $m_1(w') \ge m_2(w') \ge \ldots$ \.  For example, $w(A)$ as above
is not a ballot sequence as can be seen for the prefix $w'=112133$, with
$m_2(w')=1 < m_3(w')=2$.

\begin{thm}[Blasiak]\label{thm:blasiak}
Let $\la, \mu \vdash n$ and $\nu=(n-k,1^k)$.  Then $g(\la,\mu,\nu)$ is equal to the
number of small bar semi-standard tableaux~$A$ of shape~$\la$, content~$\mu$,
with~$k$ barred entries, such that $w(A)$ is a ballot sequence,
and such that the lower left corner of $\vk(A)$ is unbarred.
\end{thm}

Although we never defined tableaux switching, as originally defined in~\cite{BSS}
only gives a pseudo-polynomial time algorithm.  Below we show that one can
speed up the tableaux switching to be able to check the validity of
Blasiak's tableaux in polynomial time (cf.~\cite{PV2}).

\begin{proof}[Proof of Theorem~\ref{thm:hooks}]
First, we check whether $\ell(\nu)=k+1$ is greater than $\ell(\la)\ell(\mu)$.
It follows from~\cite{Dvir} that in the case we have $g(\la,\mu,\nu)=0$.

Now, if this is not the case, we must have $k \leq \ell(\la)\ell(\mu)$,
so~$k$ is polynomial in the size of the input. Proceed as follows.

The polynomial witness for \KP(hook) ~is a small-bar semi-standard tableaux $A$ of shape~$\lambda$, content~$\mu$ and $k$ barred entries.
We encode the tableaux as a pair of one $\ell(\mu) \times \ell(\la)$ matrix and one $2 \times k$ matrix. In the first matrix, the entry at position $(i,j)$ records the number of $i$'s in row~$j$. In the second matrix,  each column corresponds to a single barred letter and contains the coordinates of its position in the tableaux. This encoding requires $\ell(\mu)\ell(\la)\log N +  2k\log N$ bits. As noted above, this is polynomial in the input size.

We show that tableau-switching can be done in polynomial time. Indeed, switching a single $\ov i$ and all unbarred letters~$j$ can be done in one step, as follows. If in the process, $\ov i$ becomes adjacent to a letter $j$ for the first time then the letter $\ov i$ should move either to the row below or to the end of the horizontal strip of letters $j$.  Since there are $k=\ell(\nu)-1$ barred letters,  the tableaux-switching is done in polynomially many operations.

By Theorem~\ref{thm:blasiak}, \KN(hook) ~counts the number of tableaux as in the theorem. From the above argument, we can verify that they satisfy the condition in the Theorem in polynomial time. Moreover, the number of such tableaux is at most exponential in the input size. 
This implies the result.
\end{proof}

\bigskip

\section{Complexity of~\KN~and the proofs}\label{s:complexity_kp}

We turn towards the computational complexity of~\KN~as defined in Section~\ref{s:complexity}. Here we prove Theorem~\ref{thm:nu_bounded} and related results. We first establish our main tool for this, the Reduction Lemma.

\subsection{The Reduction Lemma}
In order to prove these statements we will need the following Lemma. Informally, it states that under the conditions of Theorem~\ref{thm:nu_bounded} we can either conclude that $g(\la,\mu,\nu)=0$ or reduce the computation of $g(\la,\mu,\nu)$ to the computation of a Kronecker coefficient for much smaller partitions. To describe these partitions and state the lemma we construct the following \emph{reduction map} $\varphi$ on partitions, which depends on
 a fixed triple of partitions $(\la,\mu,\nu)$ for which $|\lambda_i - \mu_i| \leq n-\nu_1$.

\smallskip
Let  $\ell(\la)$, $\ell(\mu)$, $\ell(\nu) \le \ell$, set $t=n-\nu_1$ and suppose that
and $|\lambda_i -\mu_i| \leq t$ for all~$i\le \ell$ (otherwise, the map~$\varphi$ is undefined).
Denote \ts $\omega = \lambda \cup \mu$ and \ts $\rho = \lambda \cap \mu$.
Let $I=\{ i: \rho_i \geq \omega_{i+1}+t, 1\leq i \leq \ell\}\cup \{\ell+1\}$, where
$\omega_{\ell+1}=0$. For all indices~$j$, set
$i_j = \min\{i \in I, i\geq j\}$ and let $ind_I(i)=\#\{ i' \in I : \ell\geq i' \geq i\}$ be the position of $i$ in $I$ when sorted in decreasing order (without counting the entry $\ell+1$) and set $ind_I(\ell+1) =1$.

For any partition $\theta$ with at most $\ell$ parts and with $\rho \subset \theta+(t^\ell)$, define the partition $\varphi(\theta)$  via its parts by
$$\varphi(\theta)_j \. = \. \theta_j - \rho_{i_j} +t(\ell-i_j + ind_I(i_j))\.,
1 \le j \le \ell\ts.
$$
Let $r=|\varphi(\la)| = |\varphi(\mu)|$, where the latter equality follows
by construction.  Define also \ts $\varphi(\nu) = (r-s, \nu_2,\nu_3,\ldots)\vdash r$.

\begin{lemma}[Reduction Lemma]\label{l:reduction}
Given integers $n$, $\ell$ and partitions $\lambda,\mu,\nu \vdash n$ such that $\ell(\la), \ell(\mu),\ell(\nu) \leq \ell$, let $t=n-\nu_1$.  We have the following two cases:
\begin{enumerate}
\item[(i)] If $|\lambda_i - \mu_i| > n-\nu_1$ for some $i$, then $g(\lambda,\mu,\nu)=0$,
\item[(ii)] If $|\lambda_i-\mu_i|\leq n-\nu_1$ for all $i=1,\ldots,\ell$, then $g(\lambda,\mu,\nu) = g(\shortp{\lambda},\shortp{\mu},\shortp{\nu})$. 
\end{enumerate}
Moreover we have that $|\shortp{\lambda}|=|\shortp{\mu}|=|\shortp{\nu}| \leq  2(n-\nu_1)\ell^2$.
\end{lemma}

\begin{figure}[hbtp]
\psfrag{r1}{$\rho-t^\ell$}
\psfrag{r2}{{\textcolor{magenta}{$\rho$}}}
\psfrag{r3}{$\rho-t^\ell$}
\psfrag{r4}{{\textcolor{red}{$\theta$}}}
\psfrag{p1}{$\vp(\rho)$}
\psfrag{p4}{}
\psfrag{o1}{$\omega$}
\psfrag{o2}{$\vp(\omega)$}
\psfrag{o3}{$\rho$}
\psfrag{o4}{{\textcolor{red}{\hskip-.55cm $\vp(\theta)$}}}
\includegraphics[scale=0.48]{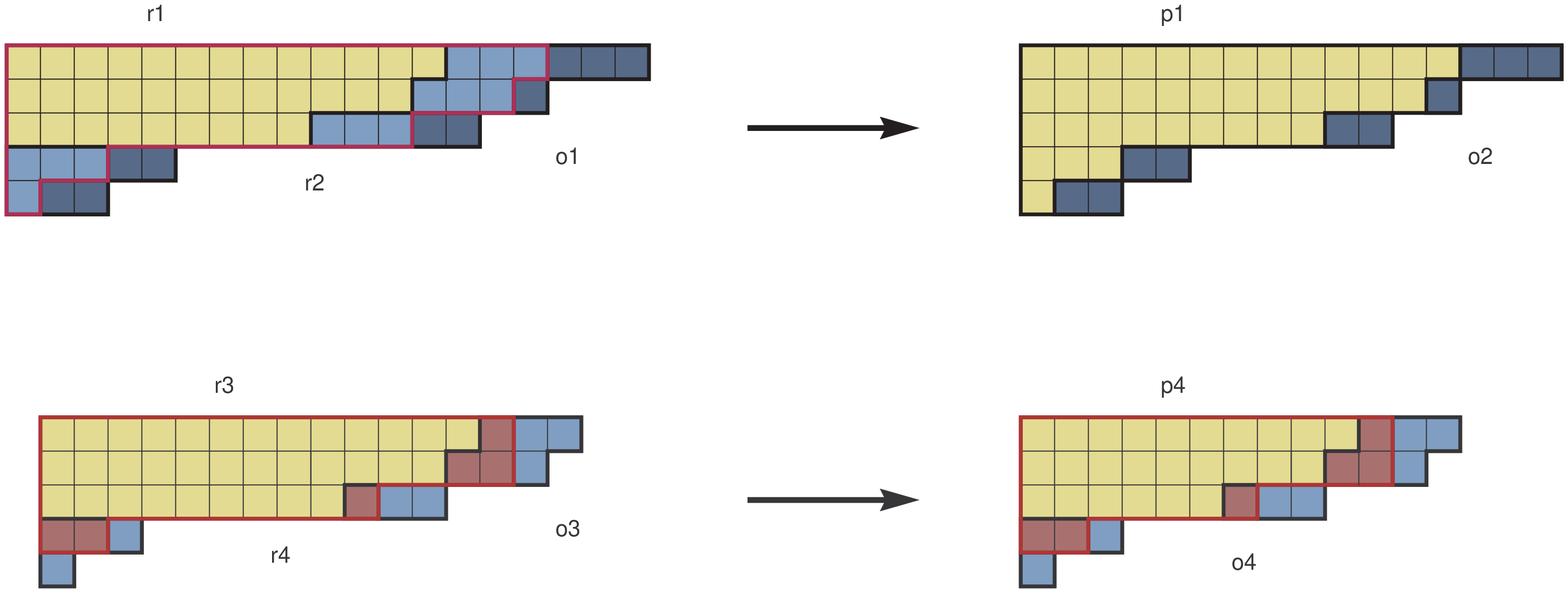}
   \caption{Example of the map $\shp$.}
\label{f:phi}
\end{figure}

\begin{ex}\label{ex:phi}{\rm
Let $\la =(19,15,12,5,1)$, $\mu =(16,16,14,3,3)$ and $t=3$. These partitions give $\omega = (19,16,14,5,3)$ and $\rho=(16,15,12,3,1)$. The relevant partitions are displayed in the top picture of Figure~\ref{f:phi}: the dark blue area is the skew shape $\omega/\rho$, whereas all blue areas represent the connected components of $\omega/(\rho-3^5)$. We have that $I=\{3,6\}$, $ind_I(3)=1$ and $ind_I(6)=1$, and we see that the map $\shp$ shifts the top 3 parts of any partition 3 boxes to the left and leaves   the bottom parts intact, giving $\shortp{\omega}= (16,13,11,5,3)$ and $\shortp{\rho}=(13,12,9,3,1)$.  Now, for $\theta =( 14,14,10,2)$, we have $\shortp{\theta} = (11,11,7,2)$, as in the figure.
}
\end{ex}

The rest of this subsection is the proof of this lemma.

\medskip

\noindent {\bf Preliminaries.} We start with a few observations and a setup for the proof.

 Clearly, by the definition of $*$, we have $s_\al*s_{n-r}=s_\al$ for all $\al \vdash n-r$.
  Then Littlewood's identity~\eqref{Littlewood}, gives
$$
[s_{\tau}] (s_{\sigma} * (s_{n-r}\ts s_{\zeta})) \, = \,
\sum_{\alpha\vdash n-r, \. \beta \vdash r} c^{\sigma}_{\al\beta} [s_\tau] \. s_{\al} (s_\beta*s_\zeta)
$$
for any three partitions  $\tau, \sigma, (n-r, \zeta) \vdash n$.
Expressing the remaining Kronecker products as Schur functions via
$$s_\beta * s_\zeta \, = \, \sum_{\gamma\vdash r} \. g(\beta,\zeta,\gamma)\ts s_\gamma\.,
$$
we conclude that the coefficient of $s_{\tau}$ can be obtained as
\begin{align}\label{mu_littlewood}
[s_{\tau}] \bigl(s_{\sigma} * (s_{n-r}\ts s_{\zeta})\bigr)
\, = \, \sum_{\alpha \vdash n-r, \. \beta \vdash r, \. \gamma \vdash r} \.
c^{\sigma}_{\al\beta} \ts c^{\tau}_{\al\gamma}\ts g(\beta,\zeta,\gamma)
\end{align}

Let $\pi=(\nu_2,\ldots,\nu_\ell)$, so that $\nu=(n-t,\pi)$.

\medskip
\noindent
{\bf Proof of part~(i).}
Suppose that $|\lambda_i - \mu_i| > t$ for some $i$ and assume without loss of generality that $\lambda_i > \mu_i$  Then we must have
$$|\lambda \cap \mu| \leq  \sum_{j \neq i} \lambda_j + \mu_i = n - \lambda_i + \mu_i < n-t. $$
In particular, there are no partitions $\al \vdash n-t, \beta \vdash t$, such that $c^{\la}_{\al\beta} \neq 0$ and $c^{\mu}_{\al\beta}\neq 0$. Otherwise by the Littlewood-Richardson rule, we must have $\al \subset \la, \mu$, so $\al \subset \la \cap \mu$ and then
$n-t = |\al| \leq |\la \cap \mu| <n-t$, contradicting the inequality above.

 Substituting  \ts $\tau\gets \mu$, $\sigma\gets \la$, $r\gets s$ and $\zeta\gets \pi$ \ts
 in equation~\eqref{mu_littlewood},  we then get
 \begin{equation}\label{eq1:sum}
[s_\mu] \left(s_\la * (s_{n-t} s_\pi ) \right)
 = \. \sum_{\alpha \vdash n-s, \. \beta \vdash s, \. \gamma \vdash s} \.
 c^{\la}_{\al\beta} \ts c^{\mu}_{\al\gamma}\ts g(\beta,\pi,\gamma)\. = 0.
 \end{equation}

The Pieri rule~\cite[\S 7.15]{Sta} gives
$$
s_{n-s} \ts s_{\pi} \. = \. s_{(n-s,\pi)} \. + \. \sum_{\eta\in P_{\nu}} \. s_{\eta}\.,
$$
where $P_{\nu}$ is the set of partitions $\eta$ obtained by adding a horizontal strip of length $n-s$ to $\pi$ (i.e. adding $n-s$ boxes to Young diagram of $\pi$, so that no column obtains more than one box)  and such that $\eta \neq \nu$. Taking Kronecker product with $s_\la$, extracting the coefficient of $s_\mu$ and comparing with equation~\eqref{eq1:sum} gives
\begin{equation}\label{mu_pieri}
0=[s_{\mu}] (s_{\la} * (s_{n-s}\ts s_{\pi})) \, = \, g(\la,\mu,\nu) \. + \. \sum_{\eta \in P_{\nu}} \. g(\la,\mu,\eta)\..
\end{equation}
Since $g(\la,\mu,\eta), g(\la,\mu,\nu) \geq 0$, they must all be equal to 0, so in particular $g(\la,\mu,\nu)=0$.
%

\medskip

\noindent{\bf Proof of part~(ii).}
Suppose now that $|\la_i-\mu_i| \leq t$ for all $i$. As in the definition of $\vp$ above, let $\omega =\la \cup \mu$ and $\rho = \la \cap \mu$. Then we have $\rho \subseteq \omega$ and $\omega_i-\rho_i \leq t$.  The idea of $\vp$ is to shift to the left the connected components in the Young diagram of $\omega/\rho$ and then perform the same shifts of parts in $\theta$ for any partition $\theta$, so that the resulting partitions are much smaller, but the skew shape $\theta/(\rho-t^\ell)$ is preserved under $\shp$. The set $I$ indicates the rows where the skew shape $\omega/(\rho-t^\ell)$ is disconnected and is used to divide the skew shapes in the corresponding connected components.

Observe that $\vp(\theta) = \theta - \rho+\vp(\rho)$ and so given any partition $\al$ of at most $\ell$ parts, we can construct the inverse $\vp^{-1}(\al) = \al+\rho - \vp(\rho)$. Since $\rho-\vp(\rho)$ is also a partition, the map $\vp$ is then a bijection from the set of partitions containing $\rho-(t^\ell)$ to all partitions.

Let $r$ be as in the preliminaries and consider equation~\eqref{mu_littlewood}. The idea is to first reduce the Littelwood-Richardson coefficients appearing there. Let $r\leq t$. We have $\alpha \vdash n-r$ and $\beta, \gamma \vdash r$. Observe that it suffices to consider only nonzero terms in the rhs of~\eqref{mu_littlewood}. Then, the condition $c^{\la}_{\alpha\beta}c^{\mu}_{\alpha\gamma} > 0$ implies that $\alpha \subseteq\la\cap\mu=\rho$. Since $|\rho|-|\alpha|\leq |\rho|-n+r\leq r
\leq t$, we also have $\rho \subseteq\alpha + (t^\ell)$ and we can construct $\shortp{\alpha}$.

Note that $c^{\la}_{\alpha\beta}$ is equal to the number of Littlewood-Richardson tableaux of shape $\la/\alpha$ of type~$\beta$. Note that the shapes $\la/\al$ and $\shortp{\la}/\shortp{\al}$ have identical connected components. This follows from the fact that in the horizontal shifts performed by $\vp$ at the rows in $I$ the component ending in row $i_j$ is not shifted beyond the end of the component right below, since
$$\vp(\al)_{i_j} - \vp(\la)_{i_j+1}  = \underbrace{\al_{i_j} - \rho_{i_j}}_{\geq -t}  - \underbrace{(\la_{i_j+1} -v\rho_{i_{j+1}})}_{\leq t(i_j - i_{j+1}) }  +t(i_{j+1}-i_j) +t \underbrace{(ind_I(i_j) -ind_I(i_{j+1}) )}_1 \geq 0.$$
Then the skew tableaux $\la/\al$ and $\vp(\la)/\vp(\al)$ are indentical, so
 we have the same number of LR tableaux of fixed type and thus
$$c^{\la}_{\al\beta}=c^{\shortp{\la}}_{\shortp{\al}\beta} \qquad \text{ and }\qquad c^{\mu}_{\al\gamma}=c^{\shortp{\mu}}_{\shortp{\al}\gamma}\, .$$

Similarly, using the inverse of $\vp$, when $|\theta| \geq  |\vp(\la)| - t$, then
$$c^{\vp(\la)}_{\theta \be} = c^{\la}_{\vp^{-1}(\theta)\be} \qquad \text{ and } c^{\vp(\mu)}_{\theta\gamma} = c^{\mu}_{\vp^{-1}(\theta)\gamma}.$$
%

Applying $\shp$ and the above observations, we can now rewrite~\eqref{mu_littlewood} with $\tau\gets \mu$, $\sigma \gets \la$, $\zeta \gets \eta$ and then apply it again for $\tau\gets \shortp{\mu}$, $\sigma \gets \shortp{\la}$. We have
\begin{equation}\label{mu_hat}
\begin{split}
[s_\mu] \; \bigl(\; s_\la * (s_{n-r} s_\eta)\bigr) &= \sum_{\al \vdash n-r, \; \beta \vdash r, \; \gamma \vdash r}  \;\; c^{\la}_{\al\beta} \;\; c^{\mu}_{\al\gamma} \; g(\beta,\eta,\gamma) \\
&=\sum_{\al \vdash n-r, \; \beta \vdash r, \; \gamma \vdash r}\;  c^{\shortp{\la}}_{\shortp{\al}\beta}\; \; c^{\shortp{\mu}}_{\shortp{\al}\gamma} \;\; g(\beta,\eta,\gamma) \\&=
\sum_{\theta \vdash \shortp{n-r},\;  \beta \vdash r,\; \gamma \vdash r}\;  c^{\shortp{\la}}_{\theta\beta}\;\;c^{\shortp{\mu}}_{\theta\gamma}\;\; g(\beta,\eta,\gamma)\\& =\;\;  [s_{\shortp{\mu}}]\;\bigl(\; s_{\shortp{\la}} * (s_{\shortp{n-r}}s_{\eta})\bigr) \, .
\end{split}
\end{equation}

We now show that this leads to part (ii) in the lemma. The idea is that $s_\nu$, for $\nu=(n-t,\pi)$, can be expressed as a linear combination of terms of the form $s_{n-r}s_\eta$ with $r \leq s\leq t$. We then apply~\eqref{mu_hat} to obtain the desired equality for the Kronecker coefficients.

By the Jacobi-Trudi identity, for any partition $\xi$, we have
\begin{equation}\label{nu_jacobi}
\begin{split}
s_{\xi}\; &= \;\det\left[ \; h_{\xi_i-i+j}\; \right]_{i,j=1}^\ell = \sum_{j=1}^\ell\;  (-1)^{j-1}\ts h_{\xi_1-1+j} \ts \det\left[\; h_{\xi_i-i+\ell}\;\right]_{i=2,k=1,\ell\neq j}^\ell \\
&= \sum_{j=0}^{\ell-1} \; (-1)^j \; h_{\xi_1+j}\left(\sum_{\eta \vdash s-j} c_{\eta}\; s_{\eta} \right)= \sum_{j=0}^{\ell-1}\sum_{\eta \vdash n-j}\; (-1)^j c_{\eta}\; s_{\xi_1+j}\;s_{\eta}\,.
\end{split}
\end{equation}
Here we expand the $(\ell-1)\times(\ell-1)$ determinants of homogenous symmetric functions as sums of Schur functions, so $c_{\eta}$ are the coefficients with which they appear in the sum.

Finally, applying~\eqref{mu_hat}  with $\nu=(n-s,\pi)$ and~\eqref{nu_jacobi} with $\xi \gets \nu$, we have
\begin{equation*}\begin{split}
g(\la,\mu,\nu) &= [s_\mu]\;(s_\la * s_\nu) = [s_\mu]\left\{\;s_\la* \left(\sum_{j=0}^{\ell-1}\sum_{\eta \vdash n-j} (-1)^j\; c_{\eta}\; s_{n-s+j}\; s_{\eta} \right)\right\}\\
&= \sum_{j=0}^{\ell-1}\sum_{\eta \vdash n-j} (-1)^j \; c_{\eta}\;  [s_\mu]\; \left(s_\la*(s_{n-s+j}\;s_{\eta})\right) \\
&= \sum_{j=0}^{\ell-1}\sum_{\eta \vdash n-j} (-1)^j \; c_{\eta}\; [s_{\shortp{\mu}}] \;(s_{\shortp{\la}}*\left(s_{\shortp{n-s+j}}\;s_{\eta})\right) \\
&=[s_{\shortp{\mu}}]\; (s_{\shortp{\la}} * \sum_{j=0}^{\ell-1}\sum_{\eta \vdash n-j} (-1)^j\; c_{\eta} \;s_{\shortp{n-s}+j}\; s_{\eta}) \\
&= [s_{\shortp{\mu}}]\; \left(s_{\shortp{\la}} * s_{(\shortp{n-s},\pi)} \right) \,
= \, g\bigl(\shortp{\mu},\shortp{\la},\shortp{\nu}\bigr)\..
\end{split}\end{equation*}
Here  the penultimate equality comes from application of~\eqref{nu_jacobi} with $\xi_1 =\shortp{n-s}$. The last equality follows from  the fact that $\shortp{\nu}=(\shortp{n-s},\pi)$, since  $|\pi|\leq t$. This completes the proof of the lemma. \qed

\medskip

\subsection{Proofs of complexity results.}\label{ss:thm_proof}
The following result (Main Lemma) gives a bound on the computational complexity of Kronecker coefficients in the general case. Together with the Reduction Lemma it gives Theorem~\ref{thm:nu_bounded}. Incidentally, its proof can be used to derive Theorem~\ref{thm:gapp}, as we explain.

Before we proceed, we need the following technical result.
For integer vectors $a,b,c \in \mathbb{Z}_{\geq 0}^\ell$, denote by $C(a,b,c)$
the number of \emph{three dimensional contingency arrays} (3-way statistical tables) with 2-way marginals $a, b, c$.

\begin{lemma}\label{l:contingency}
In the notation above, let $R=\max\{a_i,b_i,c_i\, | \, i=1,\ldots,\ell\}$.
Then the number~$C(a,b,c)$ can be computed in time \ts
$(\ell\log R)^{O(\ell^3\log \ell)}$.
\end{lemma}
\begin{proof}

Note that these contingency arrays are just the integer points in a polytope in $\mathbb{R}^{\ell^3}=\{(\ldots,x_{ijk},\ldots)|\; i,j,k=1,\ldots,\ell \}$, defined by the $\ell^3$ inequalities
 $$
 x_{ijk}\geq  0 \, , \; \text{for} \; 1\leq i,j,k\leq \ell
 $$
 and the $3\ell$ equations
 $$
 \sum_{j,k} x_{ijk} =a_i \;, \qquad  \sum_{i,k} x_{ijk} =b_j \; , \qquad \sum_{i,j} x_{ijk} =c_k \;.
$$
Recall Barvinok's algorithm (Theorem~\ref{t:barv}), which computes the number of integer points in
a polytope of dimension $d$ and input size $L$ in time~$L^{O(d\log d)}$.
Here we have $L = O(\ell^3\log R)$ and $d = O(\ell^3)$, which gives the result.
\end{proof}

\begin{lemma}[Main Lemma]\label{prop:g_poly}
Let $\al,\beta,\gamma \vdash n$ be partitions of the same size, such that \ts
$\al_1,\beta_1,\gamma_1 \leq m$ \ts and \ts $\ell(\al),\ell(\beta),\ell(\gamma) \leq \ell$.
Then \ts $g(\al,\beta,\gamma)$ \ts can be computed in time \ts $(\ell\log m)^{O(\ell^3\log\ell)}$.
\end{lemma}

\begin{proof}
Consider three sets of variables, $x=(x_1,\ldots,x_{\ell})$, $y=(y_1,\ldots,y_{\ell})$ and $z=(z_1,\ldots,z_{\ell})$.
We use the determinantal formula for Schur functions:
$$s_{\theta}(x_1,\ldots,x_\ell)= \frac{ \det \left[ \; x_i^{\theta_i+\ell-j}\; \right]_{i,j=1}^{\ell} }{\Delta(x)}\, ,$$
where $\Delta$ is the Vandermonde determinant.

Thus, for a symmetric function $F(x_1,\ldots,x_\ell)$, we have
$$
[s_{\theta}]\; F=  \bigl[\; x_1^{\theta_1+\ell-1}\cdots x_\ell^{\theta_\ell}\; \bigr] \;\; \Delta(x)\;F(x)\.,
$$
i.e.~the coefficient of $s_\theta$ in the expansion of $F$ as a linear combination of Schur functions is
equal to the coefficient of the respective monomial in the expansion of the polynomial $\Delta(x)F(x)$ as a sum of monomials.
 Applying this to identity~\eqref{extended_identity} and using the staircase partition $\delta=(\ell-1,\ldots,1,0)$, gives
\begin{equation}\label{coeffs}
\begin{split}
g(\alpha,\beta,\gamma) = [s_{\alpha}(x)s_{\beta}(y)s_{\gamma}(z)] \; \prod_{i,\;j,\; k} \frac{1}{1-x_iy_jz_k} = [x^{\alpha+\delta} y^{\beta+\delta} z^{\gamma+\delta} ] \; \Delta(x)\Delta(y)\Delta(z)\frac{1}{1-x_iy_jz_k}\\
=[x^{\alpha+\delta} y^{\beta+\delta} z^{\gamma+\delta} ]  \prod_{1\leq i<j\leq \ell} (x_i-x_j) (y_i-y_j)(z_i-z_j) \prod_{i,j,k=1}^{\ell} (1+x_iy_jz_{k} +\cdots +(x_iy_jz_k)^{m+\ell}).
\end{split}
\end{equation}
Here we truncated the infinite sums $(1-x_iy_jz_k)^{-1}$ to the maximal powers which could be involved in computing the necessary coefficients.

Using the notation from Lemma \ref{l:contingency} we have that
$$\prod_{i,j,k=1}^{\ell} (1+x_iy_jz_{k} +\cdots +(x_iy_jz_k)^{m+\ell}) = \sum_{a,b,c \in \{0,1,\ldots,m+\ell\}^\ell} C(a,b,c).$$

 Equation~\eqref{coeffs} then gives
\begin{equation}\label{gapp_eq}
g(\alpha,\beta,\gamma) \, =  \. \sum_{\sigma^1,\sigma^2,\sigma^3 \in S_\ell} \. \sgn(\sigma^1\sigma^2\sigma^3) \. C(\alpha+1-\sigma^1,\beta+1-\sigma^2,\gamma+1-\sigma^3)\.,
\end{equation}
where the sum goes over triples of permutations on $[1,\ldots,\ell]$ and
$$\alpha+1-\sigma^1 = (\alpha_1+1-\sigma^1_1,\ldots,\alpha_\ell+1-\sigma^1_{\ell})\ts,
$$
and similarly for the other terms. Applying Lemma~\ref{l:contingency} for each of the summands
in the above sum, we get that the Kronecker coefficient is computed in time
$$\bigl(\ell^2 \log(m+\ell)\bigr)^{O(\ell^3\log\ell)} \. (\ell!)^3 \, = \,
(\ell\ts \log m)^{O(\ell^3\log\ell)},$$
as desired.
\end{proof}

\medskip

\begin{proof}[Proof of Theorem~\ref{thm:gapp}]
Notice that~\eqref{gapp_eq} can be presented as the following difference: the number of contingency 3d arrays with marginals of the form $\alpha+1-\sigma^1,\beta+1-\sigma^2,\gamma+1-\sigma^3$, where $\sgn(\sigma^1\sigma^2\sigma^3)=1$, minus the number of contingency 3d arrays with marginals of the form $\alpha+1-\sigma^1,\beta+1-\sigma^2,\gamma+1-\sigma^3$, where $\sgn(\sigma^1\sigma^2\sigma^3)=-1$. Since each of these two numbers counts polynomially verifiable objects, and there are $(\ell!)^3/2$ many of them, we conclude the numbers  are in~\SP. This implies the result.
\end{proof}

\medskip

\begin{proof}[Proof of Theorem~\ref{thm:nu_bounded}]
Let $n=|\nu|$ and $t = n-\nu_1$. Then we have $t \leq \ell M$.
In at most $\ell$ steps we can check whether $|\la_i-\mu_i| >t$ for some $i$, in which case part (i) of  Lemma~\ref{l:reduction}  immediately implies that $g(\la,\mu,\nu)=0$.
Otherwise, part (ii) of the Reduction Lemma~\ref{l:reduction} implies that $g(\la,\mu,\nu)=g(\shortp{\la},\shortp{\mu},\shortp{\nu})$. Note that, by the definition of $\varphi$, we can compute $\shortp{\theta}$ in at most $3\ell$ steps.

In the Main Lemma~\ref{prop:g_poly}, let $\alpha= \shortp{\la}, \beta =\shortp{\mu},\gamma=\shortp{\nu}$ and  $m \leq 2t \ell^2\leq 2M\ell^3$. Then the lemma implies that the Kronecker coefficients can be found in time
$$
O(\ell\ts\log N) \. + \. (\ell^2 \log M)^{O(\ell^3\log\ell)},
$$
where the term $O(\ell\ts\log N)$ comes from the initial comparison of the parts of the partitions and the application of~$\varphi$.
\end{proof}

\bigskip

\section{Partitions of fixed lengths}\label{ss:fixed_length}

Here we consider the Kronecker Positivity problem \KP~when $\ell$ is fixed.
While the Main Lemma~\ref{prop:g_poly} already gives that \KP~$\in$~\P~in this case,
we present a different proof and a new complexity bound for this result.

\begin{thm}\label{thm:fixed}
The problem \KP{\rm $(\ell)$}~$\in$~\P~for every fixed~$\ell$. Moreover, for
$\la_1,\mu_1,\nu_1\le N$ and $\ell(\la),\ell(\mu),\ell(\nu)\le \ell$,
there is an algorithm which decides whether \ts $g(\la,\mu,\nu)>0$
\ts in time $O(\log N)$.
\end{thm}

\begin{proof}[Proof of Theorem~\ref{thm:fixed}]
Let $v_1,\ldots,v_k \in \mathbb{Z}_+^{3\ell}$ be the basis of generators of the Kronecker semigroup~$\SM_\ell$ (see Section~\ref{ss:semigroup}). Let $A$ be the $(3\ell)\times k$ matrix whose columns are the vectors $v_1,\ldots,v_k$, i.e. $A= [v_1 \ldots v_k]$.  Deciding whether $g(\la,\mu,\nu)>0$ is equivalent to deciding whether the vector $X=(\la,\mu,\nu)^T$ is a nonnegative integral combination of the vectors $v_1,\ldots,v_k$. In other words, we need to decide whether $AY = X$ has a solution $Y \in \mathbb{Z}_+^{3\ell}$. Since $\ell$ and $k$ are fixed constants, the matrix $A$ is also fixed. The first part of the theorem follows from Lenstra's theorem stating that this integer linear program (ILP) can be solved in polynomial time in the size of the input (i.e.~$X$) for every fixed dimension (see e.g.~\cite{Sch}).  The second part follows from more recent results on the complexity of feasibility of~ILP (see~\cite{Eis}).
\end{proof}

\begin{rem}  {\rm
Observe that the dependence on $(\log N)$ in the theorem is linear, rather than polynomial of
degree $(\ell^3)$ in the Main Lemma~\ref{prop:g_poly}, making this algorithm more efficient.
Note also that the proof of Corollary~\ref{c:semi} of the finite generation of~$\SM_\ell$
given in~\cite{CHM} is inexplicit.
Thus we have no control over the size $k=k(\ell)$ as in the proof, not even whether this is a
computable function of~$\ell$.  Getting any bounds in this direction would be interesting,
as it would make effective the constant implied by the $O(\cdot)$ notation.
}
\end{rem}

\bigskip

\section{Complexity of deciding whether a character is zero}\label{s:char}

We now consider an analogue of the \KP~problem, which we call~\CP. Note that complexity of
characters has been studied before, see~\cite{Hel} for a treatment of the problem when the
input is in unary and a conclusion that the decision problem is then PP-complete.

\medskip

\noindent
{\sc Is the character $\chi^{\la}[\nu]=0$ (\CP):}

\smallskip

\noindent
{\bf Input:} \, Integers $N, \ell$, partitions $\la=(\la_1,\ldots,\la_\ell)$,
\noindent
 $\nu=(\nu_1,\ldots,\nu_\ell)$, where $0\le \la_i, \nu_i \le N$, and $|\la|=|\nu|$.

\smallskip

\noindent
{\bf Decide:} \, whether $\chi^{\la}[\nu]=0$.

\bigskip

The main result of this section is the following theorem.

\begin{thm}\label{t:char}
The problem \CP~is \NP--hard.
\end{thm}

The proof follows from the following observation, implying that $\CP$~is at least
as hard as the Knapsack problem.

\begin{prop}
Suppose that $\ell(\la) =2$ and $\nu_i~\equiv 0 \pmod 2$. Then \CP~is \NP--hard.
\end{prop}

\begin{proof}
We reduce to the classical \NP-complete \emph{Knapsack problem}:

\textsc{Knapsack:} Given the input $(k,a_1,\ldots,a_\ell)$,
determine whether there are $\epsilon_i \in \{0,1\}$ for $i=1,\ldots,\ell$, such that
$$k= \sum_{i=1}^\ell \epsilon_i a_i.$$

Consider the  \CP problem in the special case when $\la=(n-2k,2k)$ is a two-row partition and set $\nu=(2a_1,2a_2,\ldots)$ (assume the sequence is weakly decreasing, otherwise it can be sorted).
 Frobenius' formula (Jacobi-Trudi identity) gives
$$\chi^{\la} = \chi^{(n-2k)} \circ \chi^{(2k)} - \chi^{(n-2k+1)}\circ \chi^{(2k-1)}=\chi^{(n,2k)/(2k)} - \chi^{(n,2k-1)/(2k-1)}.$$
We evaluate the characters for skew shapes by the usual Murnaghan--Nakayama rule. In this case the height of each border strip has to be 1 to fit into skew tableaux consisting of disconnected rows. For any multiset of positive integers $R=\{r_1,\ldots,r_q\}$ and any integer $m$, denote by $\pa_R(m)$ the number of ways to write $m$ as a sum of entries from $R$. In other words
$$\pa_R(m) = \# \{ (i_1,i_2,\ldots) \; : \; 1\leq i_1 < i_2 < \cdots \leq q, \text{ s.t. } r_{i_1}+r_{i_2}+\cdots = m\}.$$
 For any $a$ and $b$ the Murnaghan--Nakayama rule gives
$$\chi^{(a+b,a)/(a)}[\nu] = \pa_\gi(a),$$
where $\gi=\{\nu_1,\ldots,\nu_\ell\}$. Hence
$$\chi^{\la}[\nu] = \pa_\gi(2k) - \pa_\gi(2k-1).$$
Since all elements in $\gi$ are even, we have that $\pa_\gi(2k-1) =0$, so $\chi^{\la}[\nu]=0$ if and only if $\pa_\gi(2k)=0$. Determining whether $\pa_\gi(2k)=0$ is the same as the \textsc{Knapsack} problem above.

%
\end{proof}

\bigskip

\section{Final remarks}\label{s:final}

\subsection{} \label{ss:fin-strong}
As mentioned in the introduction, the Main Lemma~\ref{prop:g_poly} implies
that one can compute Kronecker coefficients in time polynomial in the size
of the parts, but exponential in the number of parts.  This phenomenon is
similar to the well known distinction between \emph{weak} and \emph{strong}
\NP-completeness (see e.g.~\cite{GJ,Pap,Vaz}), corresponding to the input given in
\emph{binary} and in \emph{unary}. It applies to other counting problems as well.
For example, counting the number of perfect matchings in
graphs with large multiple edges and fixed number~$n$ of vertices is easily
polynomial, while for large~$n$ the problem is classically \SP-complete even when
edge multiplicities are $0$ or~$1$.  More generally, the  number of \textsc{BIN PACKING}
solutions is strongly \SP-complete, via a standard reduction to
\textsc{\#\ts\/TRIPARTITE MATCHINGS} problem (see e.g.~\cite{Pap}).
Of course, for other counting problems this phenomenon fails.  For example,
the counting of \textsc{KNAPSACK} solutions is polynomial when the input is in unary.

From this point of view, we believe that the bounds in the Main Lemma cannot be
substantially improved.

\begin{conj}
The Kronecker coefficients $g(\la,\mu,\nu)$
and the LR~coefficients $c^\al_{\be\ga}$ are strongly $\SP$-hard.
\end{conj}

Of course, the second part of the conjecture implies the first part, via
Murnaghan's reduction (see below).  The final reduction in~\cite{Nar}
proving $\SP$-completeness of LR~coefficients, is from \emph{contingency
tables} (see~\cite{DG} for an introduction).  It is easy to see that the decision
problem for existence of contingency tables with given marginals is polynomial
when the input is in unary, and $\NP$-complete when the input is in binary.
However, despite the large body of literature, it seems open whether the counting
problem is strongly $\SP$-complete.  Note, however, that De~Loera and Onn proved
that for the three-way statistical tables (see~$\S$\ref{ss:thm_proof}),
the counting problem is strongly $\SP$-complete, even when one dimension
$m\ge 2$ is bounded, see~\cite{DO}.

\subsection{}
The Littlewood--Richardson coefficients are much better understood than the Kronecker coefficients.
In fact, the LR coefficients are actually a special case of the Kronecker coefficients:
\begin{equation}\label{eq:cg}
c^{\la}_{\mu\nu} \. = \.
g\bigl( (n -|\la|,\la), (n-|\mu|,\mu),(n-|\nu|,\nu)\bigr)\ts,
\end{equation}
for any partitions $\la,\mu,\nu$, such that $|\la|=|\mu|+|\nu|$, and any sufficiently
large~$n$. This equality is due to Murnaghan and Littlewood~\cite{L2,Mur2,Mur3}.
Suppose that $\nu_1 \leq M$. Then the Reduction Lemma applied to
the partitions in~\eqref{eq:cg}, gives the following alternative for the LR~coefficients.
When $|\la_i-\mu_i|>|\nu|$ for some $i$, we have $c^{\la}_{\mu\nu}=0$;
otherwise, there exist partitions $\psi(\la),\psi(\mu)$ of sizes at most $\ell^3 M$,
such that $c^{\la}_{\mu\nu} = c^{\psi(\la)}_{\psi(\mu)\nu}$.
Note that these results can also be obtained directly from the Littlewood--Richardson
rule, as in the proof of the Reduction Lemma. Let us mention also that for $i=1$,
part~(i) of the Reduction Lemma is proved in~\cite{Kle}(see also~\cite[$\S$2.9]{JK}).

The analogue of Theorem~\ref{thm:nu_bounded} for LR~coefficients is the
following result.

\begin{cor} \label{c:LR-alg}
When $\nu_1 \leq M$, $\la_1,\mu_1\leq N$ and $\ell(\mu),\ell(\la),\ell(\nu)\leq \ell$,
the Littlewood--Richardson coefficient $c^{\la}_{\mu\nu}$ can be computed
in time \ts $O(\ell\ts\log N) \. + \. (\ell\ts \log M)^{O(\ell^3\log\ell)}.$
\end{cor}

This result seems already nontrivial and hard to establish directly. It would be nice
to improve the complexity in the corollary 
(cf.~$\S$\ref{ss:fin-barv} below).  In a different direction, it would be interesting
to extend the Main Theorem~\ref{thm:nu_bounded} to \emph{plethystic constants}~$a_{\la\mu}^\pi$
(see e.g.~\cite{JK,Mul1,Sta}).

\subsection{} An analogous to~\KP, yet much simpler, is the positivity problem
for the Littlewood--Richardson coefficients:

\smallskip

\noindent
\textsc{LRP~Problem:}
Given the input $(\la,\mu,\nu)$ as in $\KP$, decide whether $c^{\la}_{\mu\nu}>0$.

\smallskip

Knutson and Tao's proof of the \emph{Saturation Conjecture}~\cite{KT} implies
that this decision problem is in~\P, since it reduces to a feasibility
problem of a linear program with $O(\ell^2)$ inequalities and constraints of
size~$O(N)$, see~\cite{MNS,BI-flow}.  Such problems can be solved in
polynomial time in the size~$s$ of the input, $s=\Theta(\ell \log N)$.

Unfortunately, the saturation theorem does not hold for Kronecker
coefficients (see e.g.\/ \cite[$\S 2.5$]{Kir}).
Mulmuley's original approach to \KP~was via Conjecture~\ref{conj:sp-pol}
and a weak version of the Saturation Conjecture.  The original version
of the latter was disproved in~\cite{BOR}, and modified by Mulmuley
in the appendix to~\cite{BOR} (see also~\cite{Mul1}).

While the decision problem for the positivity of
LR--coefficients is in~{\NP} even without the Knutson-Tao theorem,
 conjectures~\ref{conj:decision}  and~\ref{conj:sp} remain out of reach.
 As of now, there are no combinatorial interpretation for Kronecker
 coefficients $g(\la,\mu,\nu)$ except for a few special cases
(see the references in~\cite{PP-future,PPV}).

\subsection{}\label{ss:fin-barv}
For the LR coefficients, one can apply Barvinok's algorithm
for counting integer points in polytopes of BZ triangles, see~\cite{DM}.
In notation of Corollary~\ref{c:LR-alg}, these polytopes have dimension
$d=\theta(\ell^2)$ and input size $L=O(\ell \log N)$.
By Theorem~\ref{t:barv}, the resulting algorithm has cost
$$
L^{O(d\log d)} \, = \, \bigl(\ell \log N\bigr)^{O(\ell^2\log \ell)}\..
$$
This is roughly comparable with the result of Corollary~\ref{c:LR-alg},
larger in some cases and smaller in other.

Note also that in light of Barvinok's algorithm, one can view the
Main Lemma~\ref{prop:g_poly} as an evidence in support of
Mulmuley's Conjecture~\ref{conj:sp-pol}.


\subsection{}
A positive combinatorial interpretation for the Kronecker coefficients,
analogous to the LR--rule, would likely show that the decision problem
is in~\NP~and the counting problem in~\SP. Such interpretation would also
imply a combinatorial interpretation for the difference between the number
of partitions of $k$ and the number of partitions of $k-1$, which fit into
an $\ell \times m$ rectangle (see~\cite{PP}). Formally, this difference
is equal to $g\bigl(m^\ell,m^\ell,(n-k,k)\bigr)$; in full generality
its combinatorial interpretation is already highly nontrivial and will
appear in~\cite{PP-future} (see also~\cite{BO}).

\subsection{}\label{fr:comb_kp}
The known results so far do not prove Mulmuley's Conjecture~\ref{conj:decision}, even when
the input is in unary. As the current results suggest, the computational
complexity comes from two sources -- the length~$\ell$ of the partitions,
and the size~$N$ of their parts. While it is often possible to reduce the
problem to one where the size of the parts is $O(\log N)$, the exponential
dependence on~$\ell$ cannot be reduced with current methods.
As suggested by the proof of Lemma~\ref{prop:g_poly} and the equivalent formulas
through inverse Kostka coefficients (see~\cite{Val1}), the Kronecker coefficients are
given by alternating sums over all permutations in $S_{\ell}$, whose number is~$O(\ell^\ell)$.

\subsection{} \label{ss:fin-semi}
The semigroup property (Theorem~\ref{t:semi})
was conjectured by Klyachko in~2004, and recently proved in~\cite{CHM}
(see also~\cite{Man}).  It is the analogue of the semigroup property
of LR~coefficients proved by Brion and Knop in~1989 (see~\cite{Zel}
for the history and the related results).

\subsection{} \label{ss:fin-comp}
Throughout the paper, we are rather relaxed in our treatment of the
algorithm timing.  Our time complexity is in the cost of arithmetic
operations with integers as in the input.

\subsection{} \label{ss:fin-red}
The \emph{reduced Kronecker coefficients}, see e.g.~\cite{BDO,BOR2}, are defined as
\begin{equation}\label{eq:red}
\bar{g}(\la,\mu,\nu) \. = \. g\bigl((n-|\la|,\la),(n-|\mu|,\mu),(n-|\nu|,\nu)\bigr) \quad \text{for $n$ large enough.}
\end{equation}
Note that equation~\eqref{eq:red} generalizes equation~\eqref{eq:cg}
when there is no constraint $|\la| = |\mu|+|\nu|$.
The fact that $\bar{g}(\la,\mu,\nu)$ are well defined has been established by Murnaghan~\cite{Mur}
(see also~\cite{Brion,Val1}), but determining effective bounds on $n$ for which the sequence
stabilizes is still an active area. The Reduction Lemma~\ref{l:reduction} immediately implies
Murnaghan's result that they stabilize. Its proof, and more specifically the map $\shp$,
also gives the following upper bound.

\begin{cor}
Equation~\eqref{eq:red} holds for all triples of partitions $(\la,\mu,\nu)$,
such that $|\nu|\le|\mu|\le |\la|$ and $n\le \max\{\la_1,\mu_1\} +|\nu|+|\la|$.
\end{cor}

This result matches the bound in \cite{Val1} in the cases when $\la_1\geq \mu_1$,
but is slightly weaker otherwise. This result is also comparable to the result of~\cite[Theorem 1.4]{BOR2}, where, for example, in the case $\la=\mu=\nu$ they coincide, but in general is also slightly weaker.

\begin{proof}[Sketch of proof.]
Apply the map $\shp$ from Lemma \ref{l:reduction} with $t=|\nu|$ to the partitions $(n-|\la|,\la)$ and $(n-|\mu|,\mu)$. When $n \geq \max\{\la_1,\mu_1\} +|\nu|+|\la|$, we have  $1 \in I$. Then, for any partition~$\theta$ as in the proof, the first part~$\shortp{\theta}_1$ does not depend on $n$ anymore. Thus  $\shortp{n-|\la|,\la}$, $\shortp{n-|\mu|,\mu}$, $\shortp{n-|\nu|,\nu}$ are also independent of~$n$. The rest follows from the proof of Lemma~\ref{l:reduction}.
\end{proof}

\smallskip

It would be interesting to see if the Reduction Lemma can be further extended to
imply better bounds, or whether there is a plethystic generalization.

\subsection{} \label{ss:fin-char}
Note that the elementary construction in the proof of Theorem~\ref{t:char} also implies
that computing characters is $\SP$-hard, a result obtained earlier in~\cite{Hel}.
In~\cite{PPV}, the authors prove that positivity of certain Kronecker coefficients is a consequence
of nonzero character values, which are easier to establish via the Murnaghan--Nakayama rule in
certain large cases.
Unfortunately, the fact that \CP~is \NP--hard (Theorem~\ref{t:char}), implies that this
approach is unlikely to have complexity implications.  Similarly,  using characters to
efficiently compute the Kronecker coefficients via the formula in~$\S$\ref{ss:def-kron},
is most likely going to be futile (cf.~\cite[$\S$13.5]{BCS}).

\bigskip

\noindent
{\bf Acknowledgements.} \
We are grateful to Abdelmalek Abdesselam, Jonah Blasiak, Emmanuel Briand, Jin-Yi Cai,
Jes{\'u}s De Loera, Allen Knutson, Joseph Landsberg, L\'{a}szl\'{o} Lov\'{a}sz,
Rosa Orellana and Bernd Sturmfels for interesting conversations and helpful remarks.
We are especially thankful to Christian Ikenmeyer and Ernesto Vallejo for reading
a draft of this manuscript and help with the references, and anonymous referees 
for careful reading of the previous version of the paper.  The first author gives
special thanks to Alexander Barvinok for the generosity in explaining his
eponymous algorithm many years ago and clarifying its complexity issues.
The first author was partially supported by the NSF~grant, the second
by the Simons Postdoctoral Fellowship.


\vskip1.4cm


\end{document}